\documentclass[12pt, reqno]{amsart}
\usepackage{amsmath}
\usepackage{amssymb}

 2

\newtheorem{thm}{Theorem}[section]
\newtheorem{lem}[thm]{Lemma}
\newtheorem{prop}[thm]{Proposition}

\newcommand{\thmref}[1]{Theorem~\ref{#1}}
\newcommand{\propref}[1]{Proposition~\ref{#1}}

\theoremstyle{remark}
\newtheorem{rmk}{Remark}[section]
\newcommand{\Z}{{\mathbb Z}}
\newcommand{\Q}{{\mathbb Q}}
\newcommand{\N}{{\mathbb N}}
\newcommand{\R}{{\mathbb R}}
\newcommand{\C}{{\mathbb C}}

\begin{document}

\title[Generalized Hurwitz zeta functions]{On the zeros of 
generalized Hurwitz zeta functions}


\author{T. Chatterjee and S. Gun}

\address[T. Chatterjee]
      {Department of Mathematics, 
      Queen's University, Kingston,
      ON K7L 3N6, 
      Canada.}
\email{tapasc@mast.queensu.ca}

\address[S. Gun]
     {Institute of Mathematical Sciences,
       4th Cross Street, C.I.T Campus, 
       Taramani, Chennai 600113,
       India.}
\email{sanoli@imsc.res.in}

\maketitle

\begin{abstract}
In this note, we prove the existence of  infinitely many zeros of certain generalized 
Hurwitz zeta functions in the domain 
of absolute convergence. This is a generalization 
of a classical problem of Davenport, Heilbronn and Cassels 
about the zeros of the Hurwitz zeta function.
\end{abstract}

\smallskip

\section{\bf Introduction}

\bigskip

For a real number $\alpha>0$, the Hurwitz zeta function $\zeta(s,\alpha)$ 
is defined by
\begin{equation*}
\zeta(s,\alpha):= \sum_{n=0}^\infty \frac{1}{(n+\alpha)^s},
\end{equation*}
where $s\in \C$ with real part $\Re(s)> 1$. It has a meromorphic continuation
to the complex plane $\C$, its only pole being a
simple pole at $s=1$ with residue~$1$. 
From now on, we set $s := \sigma +  i t$
with $\sigma, t \in \R$.

In a classical paper, Davenport and Heilbronn \cite{DH} proved that if 
$\alpha \ne 1/2, 1$ is either rational or transcendental, then
$\zeta(s,\alpha)$ has infinitely many zeros with $\sigma >1$. 
Since $\zeta(s,1) = \zeta(s)$ and $\zeta(s, 1/2) = (2^s -1) \zeta(s)$,
they do not have zeros with $\sigma > 1$. On the other hand
when $\alpha$ is an 
algebraic irrational, Cassels \cite{JC} showed 
the existence of infinitely many zeros of $\zeta(s,\alpha)$ for $\sigma >1$.  
For an illuminating account of this theme, see the recent article
of Bombieri and Ghosh \cite{BG}.

For a periodic arithmetic function $f$ with period $q \ge 1$ and $\alpha > 0$, consider
the $L$-function
\begin{equation*}
L(s,f, \alpha): =\sum_{n=0}^\infty \frac{f(n)}{(n+\alpha)^s}, 
\end{equation*}
where $s\in \C$ with $\sigma>1$. This can be thought of as a generalization
of the Hurwitz zeta function. In a recent work, Saias and 
Weingartner \cite{SW} showed that $L(s, f, \alpha)$ has infinitely
many zeros for $\sigma >1$ when $\alpha =1$ and
$L(s,f,1)$ is not  a product of  $L(s, \chi)$ and a Dirichlet
polynomial, where $\chi$ is a Dirichlet character.
In this paper, we study the zeros of $L(s, f,\alpha)$ for 
an arbitrary positive irrational number $\alpha$
in the region $\sigma >1$.
See also a recent work of Booker and Thorne \cite{BT}
where another  generalization of the work of Saias
and Weingartner to higher degree $L$-functions
 is considered.

\smallskip

Since $f$ is periodic with period $q \ge 1$, the generalized 
Hurwitz zeta function can be written as
\begin{equation*}
L(s, f, \alpha) = q^{-s}\sum_{b=1}^q f(b)\zeta(s,(\alpha+b)/q), 
\end{equation*}
for $s\in \C$ with $\sigma>1$. This shows that $L(s, f, \alpha)$ extends 
meromorphically to the whole complex plane with a possible simple 
pole at $s=1$ with residue $q^{-1}\sum_{\substack{b=1}}^{q} f(b)$.
Moreover, if we assume that $f$ is real valued
and $L(s, f, \alpha)$ does have a pole at $s=1$, we  then have
the following theorems about the existence of zeros of $L(s,f,\alpha)$.

\begin{thm}\label{thm1}
Let $\alpha$ be a positive transcendental number and 
$f$ be a real valued periodic arithmetic function 
with period $q \ge 1$. If $L(s, f, \alpha)$ has a pole at $s=1$, then
$L(s, f, \alpha)$ has infinitely many zeros for $\sigma>1$.
\end{thm}

\begin{thm}\label{thm2}
Let $\alpha$ be a positive algebraic irrational number and 
$f$ be a positive valued periodic arithmetic function 
with period $q \ge 1$. Also let $c$ be defined by
$$
c:= \frac{\underset{n}{\rm max}~f(n)}{\underset{n}{\rm min}~f(n)} .
$$
Assume that $c< 1.15$ and $L(s, f, \alpha)$ has a pole at $s=1$. Then
$L(s, f, \alpha)$ has infinitely many zeros for $\sigma>1$.
\end{thm}

\smallskip

\section{\bf Notations and Preliminaries}

\medskip 

From now on, we denote the field of algebraic numbers, the
multiplicative group of non-zero real numbers, the set of 
positive real numbers and the set of non-negative integers
by $\overline{\Q}$, $\R^*$, $\R_+$ and $\N$ respectively.

The following theorem of Kronecker (see \cite{KC}) will play an important role
in proving \thmref{thm1}.

\begin{thm}{\rm(Kronecker)}.\label{thm3}
Let $\alpha_1, \cdots, \alpha_N$ be real numbers which are linearly independent over the
integers and $\beta_1, \cdots ,\beta_N$ be arbitrary real numbers. Then for 
any real number $T$ and $\delta > 0$, there exists a real number 
$t >T$ and integers $x_1,\cdots,x_N$ such that
$$
|t\alpha_n - \beta_n -x_n|< \delta
$$
for all $n=1, \cdots, N$.
\end{thm}

\begin{lem}{\rm(Cassels \cite{JC})}.
Let $\alpha$ be a real algebraic irrational number and $K = \Q(\alpha)$. Also let
$\mathfrak{a}$ be an integral ideal
such that $\mathfrak{a} (\alpha\mathcal{O}_K)$ is an integral ideal. 
Then there exists an $N_0 > 10^6$ satisfying the following property; \par
for any $N > N_0$ and $M = [10^{-6} N]$, there are at least
$51M/100$ integers $n$ in $ N< n \le N + M$ such that 
$(n+\alpha)\mathfrak{a}$ is divisible by a prime ideal $\mathfrak{p}_n$
for which 
$$
\mathfrak{p}_n \nmid \underset{m \le N + M \atop m \ne n} 
\prod (m + \alpha)\mathfrak{a}. 
$$
\end{lem}

\smallskip

\begin{rmk}
It is not difficult to see that the proof of Cassels 
yields at least $27M/50$ integers $n$ in $ N< n \le N + M$ 
with the above property. We will make use of this fact
in the proof of \thmref{thm2}.
\end{rmk}

\section{\bf Proofs of \thmref{thm1} and \thmref{thm2}}

\smallskip

In order to prove \thmref{thm1} and \thmref{thm2}, we shall need 
the following propositions.

\begin{prop}\label{prop1}
Let $\alpha>0$ be any transcendental number and $N \ge 1$ be an integer.  
Also, let $g_1, \cdots, g_N$ be a sequence of complex 
numbers having absolute value $1$. Then for any real 
number $T$ and $\epsilon > 0$, there exists a real number $t > T$
such that
\begin{equation*}
|(n+\alpha)^{-it} - g_n| < \epsilon 
\end{equation*}
for all $1 \le n \le N$.
\end{prop}

\begin{proof}
Since $\alpha$ is transcendental, the numbers $\log(n+\alpha) $ are linearly 
independent over $\Q$.  Write $g_n = e^{-i\alpha_n}$, where 
$\alpha_n$'s are real numbers. Let  $\delta > 0$ be arbitrary. Then 
by Kronecker's theorem, 
there exists a real number $t > T$ and integers $x_1, \cdots, x_N$ such that
\begin{equation*}
\left|t ~\frac{\log(n+\alpha)}{2\pi}- \frac{\alpha_n}{2\pi} - x_n \right|< 
\frac{\delta}{2\pi}.
\end{equation*}
Multiplying both sides by $2\pi$, we get
\begin{equation*}
|t \log(n+\alpha)- \alpha_n - 2\pi x_n |< \delta.
\end{equation*}
Hence we have
\begin{equation*}
|e^{-it\log(n+\alpha)} - e^{-i\alpha_n}|<\epsilon
\end{equation*}
since $e^{-ix}$ is continuous. This completes the proof.
\end{proof}

\smallskip

Our next proposition shows that with a little 
modification in the properties of $g_n$,
one can get a similar result as above when $\alpha$ is a
positive algebraic number. In this case, consider 
the number field ${\rm K} = \Q(\alpha)$. 
For each prime ideal $\mathfrak{p}$ in the ring 
of integers $\mathcal{O}_{\rm K}$, 
let $\chi(\mathfrak{p})$ be a complex number with 
$|\chi(\mathfrak{p})|=1$. We extend $\chi$ to any element $\gamma$ of the 
number field $\rm K$ by setting 
\begin{equation}\label{one}
\chi(\gamma) = \prod_{\mathfrak{p}} \chi(\mathfrak{p})^{\nu_{\mathfrak{p}}} 
\phantom{mm} \text{ if } \phantom{m} 
\gamma\mathcal{O}_{\rm K} = \prod_{\mathfrak{p}} \mathfrak{p}^{\nu_{\mathfrak{p}}}.
\end{equation}
Here we have the following proposition. 
\smallskip

\begin{prop}\label{prop2}
Let $N \in \N, ~\alpha \in \overline{\Q} \cap \R_+$. Also let  
${\rm K} = \Q(\alpha)$ and $\chi$ be as in \eqref{one}. 
Then for any real number $T$ and $\epsilon > 0$, there 
exists a real number $t > T$ such that
\begin{equation*}
|(n+ \alpha)^{-it} - \chi(n + \alpha)| < \epsilon 
\end{equation*}
for all $0 \leq n \leq N$. 
\end{prop}

\begin{proof}
Consider the multiplicative subgroup $\rm A$ of $\R^*$ generated by 
\begin{equation*}
{\rm S}:=  \{ n + \alpha ~|~ 0\leq n \leq N \}.
\end{equation*}
One can choose a $\Z$-basis 
${\rm B}:= \{ s_j \mid 1 \le j \le l \}$ 
of A. Then there exists 
an integer $M > 0$ such that
for any $0 \le n \le N$, we have 
$$
n + \alpha = \prod_{j=1}^l  s_j^{u_j},
$$
where $u_j \in \Z$ with $|u_j| \le M$.
This implies that
\begin{equation}\label{eq1}
\chi(n+\alpha) = \prod_{j=1}^l \chi(s_j)^{u_j}. 
\end{equation}
By Kronecker's theorem, for any $\epsilon > 0$, there 
exists a real number 
$t > T$ such that
\begin{equation}\label{eq2}
|s_j^{-it} - \chi(s_j)| <  \epsilon/ Ml,  
\end{equation}
for all $s_j \in {\rm B}$. Now for any $n + \alpha \in {\rm S}$ 
and $\epsilon > 0$, we have
by \eqref{eq1} and \eqref{eq2} that 
\begin{eqnarray*}
&& |(n + \alpha)^{-it} - \chi(n + \alpha)| \\
&=& \left| \prod_{j=1}^l {s_j}^{-itu_j}
 -  \prod_{j=1}^l \chi(s_j)^{u_j} \right|,  
\text{  where } u_j \in \Z \text{ and } |u_j| \le M \\
& \le & M \sum_{j=1}^l \left| {s_j}^{-it} - \chi(s_j)\right| 
~ < \epsilon.
\end{eqnarray*}
Hence the proposition.
\end{proof}

We now prove the following proposition which will play an important
role in the proof of  \thmref{thm2}.

\begin{prop}\label{prop3}
Let $0 < r_1 \le r_2 \le \cdots \le r_n$ be real numbers. Then the set 
$$
\Delta_n :=  \left\{  c_1 r_1 +  \cdots + c_n r_n ~|~  |c_i| =1, ~c_i \in \C \right\} 
$$
for $n \ge 1$ is a closed annulus with outer 
radius $R_n$ and inner radius
$$
T_n : = 
 \begin{cases} 
r_n  - R_{n-1}  & \text{ if   $R_{n-1} \le  r_n$,}\\
 0  & \text{ otherwise. }
 \end{cases}
$$
Here $R_0 : = 0$ and $R_i : = r_1 + \cdots + r_i$ for $1 \le i \le n$.
\end{prop}

\begin{proof}
Note that the set $\Delta_n$ is compact, connected and invariant under rotation
around the origin and hence the result follows by induction on $n$.


\end{proof}

Before we state our next proposition, we shall  formulate a
hypothesis which is integral to our proofs. 
Let $f$ be a periodic arithmetic function with period $q \ge 1$.

{\it Hypothesis: 
For any  $\delta > 0$, there exists a function
 $F(s)$ analytic in  the region $\Re(s) > 1$ satisfying the 
following properties:
\begin{enumerate}
\item 
There exists a $\sigma_0$ with $F(\sigma_0) =0$ and 
$1 < \sigma_0 < 1+ \delta$. 
\item 
For any real number $T$ and  real numbers $\epsilon, \theta >0$, 
there exists a real number $t > T$ such that
$$
|L(s +it , f, \alpha) - F(s)| < \epsilon \phantom{mm}
{\rm for~~ all} \phantom{m}  \sigma > 1+ \theta.
$$
\end{enumerate}
}

\smallskip

In this context,  we have the following proposition.

\begin{prop}\label{prop4}
Let $\alpha \in \R_{+}$ and $f$ be as above.
Assume the previous hypothesis. Then $L(s,f,\alpha)$ has infinitely 
many zeros for $\sigma > 1$. 
\end{prop}

\begin{proof}
Let $T, \delta > 0$ be real numbers. We will show that  there exists a zero
$s_1$ of $L(s,f,\alpha)$ with $1< \Re(s_1) < 1+ \delta$ and $\Im(s_1) > T$. 

Let $F(s)$ be a function corresponding to $\delta$ in the hypothesis.
By property $(1)$ of $F(s)$, it has a zero $\sigma_0$ with
$1 < \sigma_0 < 1 + \delta$.  Since $F(s)$ is an analytic function,
one can choose $\delta_1 > 0$ such that $1+ \delta_1 < 1+ \delta$ and 
$1 + \delta_1 < \sigma_0$ with $F(s) \ne 0$  for $|s - \sigma_0| = \delta_1$.  Set 
\begin{equation*}
\epsilon : =  \underset{ ~|s - \sigma_0| = \delta_1}{\text{min }}|F(s)| 
\phantom{m}  \text{  and } \phantom{m}   \theta < \sigma_0 - \delta_1 - 1. 
\end{equation*}
Then $ \sigma_0 - \delta_1 > 1+ \theta$ and hence by  property $(2)$
of $F(s)$, there exists a real number $t > T$ such that
$$
|L(s+ it, f, \alpha) - F(s)| < |F(s)|
$$ 
on $|s - \sigma_0| = \delta_1$. Thus by Rouch\'e's theorem, the function 
$L(s+it, f, \alpha)$ has a zero $s_1$  which gives a zero  $s_1 + it $
of $L(s,f,\alpha)$.
\end{proof}

\smallskip

In view of the above proposition, our task is to construct
functions $F(s)$ as described in the hypothesis.

\subsection{\bf Proof of \thmref{thm1}}

Let $\alpha$ be a positive transcendental number. 
In this case, replacing $f$ by $-f$ if needed,
we can assume that the residue 
$\frac{1}{q} \sum_{b=1}^{q} f(b)$ of $L(s,f,\alpha)$ 
is a positive real number. 

Since $L(s,f,\alpha)$ converges absolutely for $\sigma> 1$, for
any $\delta>0$, one can choose an integer $m$ such that
\begin{equation}\label{trans}
\sum_{n=0}^m \frac{f(n)}{(n+\alpha)^{1+\delta}} >  \sum_{n=m+1}^{\infty} 
\frac{f(n)}{(n+\alpha)^{1+\delta}}
\end{equation}
Define 
\begin{eqnarray*}
F(s) := \sum_{n=0}^{\infty} \frac{f(n)\beta(n)}{(n+\alpha)^s}  
\phantom{mm} \text{ for } \Re(s) > 1,
\end{eqnarray*}
where $\beta$ is the arithmetic function defined by
$$
\beta(n) := 
 \begin{cases} 
- 1 & \text{ if  $n > m$,}\\
 1 & \text{ otherwise. }
 \end{cases}
$$
By \eqref{trans}, it is clear that $F(1+ \delta) > 0$. 
On the other hand, since $L(s,f,\alpha)$ has a pole at 
$s=1$,  $F(s) \to - \infty$ as $s \to 1+$. 
Since $F(s)$ is a real valued continuous function
when $s$ is real and $s >1$,
it follows that $F(s)$ has a zero in
the interval  $(1, 1+ \delta)$. 
Thus $F(s)$ is an analytic function satisfying (1).
It follows from \propref{prop1} that given any real numbers 
$T, \epsilon> 0$, there exists a real number $t> T$
such that
$$
| L(s + it,f,\alpha) - F(s)| < \epsilon  \text{  for } \Re(s)> 1. 
$$
Since $F(s)$ is a function satisfying properties (1) and (2)
in the hypothesis, 
\thmref{thm1} follows from \propref{prop4}.

\subsection{\bf Proof of \thmref{thm2}}

Let $\alpha$ be a positive algebraic irrational number
and $f$ be a positive real valued periodic function satisfying 
the conditions of \thmref{thm2}. 
Let  $\delta > 0$ be fixed.
We  shall define for $\Re(s) > 1$, a function 
$$
F(s) = \sum_{n=0}^{\infty}\frac{f(n) \chi(n+\alpha)}{(n + \alpha)^s},
$$
where $\chi$ is a suitably chosen character on the group
of fractional ideals of \linebreak $K = \Q(\alpha)$. Here
$\chi(n+\alpha)$ is the value of this character on
the principal ideal $(n+\alpha)\mathcal{O}_K$.

Clearly, such a function is holomorphic in $\Re(s) >1$.
Furthermore, it follows from \propref{prop2}
that $F(s)$ satisfies property (2). 

We shall  show that it is possible to define $\chi$ suitably
to ensure the existence of a $\sigma$ with 
$1 < \sigma < \text{ min }(1 + \delta, 2)$ satisfying
\begin{equation}\label{algebraic}
F(\sigma)= \sum_{n=0}^{\infty}\frac{f(n) \chi(n+\alpha)}{(n + \alpha)^{\sigma}} = 0.
\end{equation}
Then this function will also satisfy property (1) and thereby
establish the theorem. The idea of constructing $\chi$ with the 
above properties  is similar to that of Cassels \cite{JC}, however
our proof admits the function $f$ which did not appear
in Cassels's paper.

We first begin by setting  $N_1 = [\text{max }(N_0,  10^7, 10^7\alpha)]$,
where $[x]$ is the integral part of a real number $x$.
Then since $L(s, f, \alpha)$ has a pole at $s =1$, there exists
a $\sigma$ such that 
\begin{equation}\label{case}
\sum_{n=0}^{N_1} \frac{f(n)}{(n + \alpha)^{\sigma}} 
< 10^{-2}  \sum_{n=N_1 + 1}^{\infty}  \frac{f(n)}{(n + \alpha)^{\sigma}}
\end{equation}
and $1< \sigma < \text{ min }(1 + \delta, 2)$. 

We now define an infinite sequence of integer 
pairs $N_j, M_j$ for $j\ge 1$ by 
$$
M_j:= \left[\frac{N_j}{10^6}\right] 
\text{     and     } 
N_{j+1}: = N_j  + M_j.
$$
To prove \eqref{algebraic}, it is sufficient to show that we can construct
a character $\chi$ such that
\begin{equation}\label{eqalg}
\left| \sum_{n=0}^{N_j} \frac{f(n) \chi(n+\alpha)}{(n + \alpha)^{\sigma}} \right|
~<~ 10^{-2}  \sum_{n=N_j + 1}^{\infty}  \frac{f(n)}{(n + \alpha)^{\sigma}}
\end{equation}
for all $j$. 

Let $\mathfrak{a}$ be the ideal denominator of $\alpha$ 
so that $\mathfrak{a} (n + \alpha)$ is an integral ideal
for every integer $n$. If $\mathfrak{p} | \mathfrak{a}$ or $ \mathfrak{p} | 
(n + \alpha)\mathfrak{a}$ for $n \le N_1$,
we choose $\chi(\mathfrak{p}):= 1$. 
Then by \eqref{case}, we see that \eqref{eqalg} is true for $j=1$. 

Suppose that \eqref{eqalg} is true for all integers $\le j$. Define two subsets 
of integers as follows:
$$
\mathfrak{A}: = \left\{ n ~|~  N_j < n \le N_{j+1},~  \exists ~\mathfrak{p}_n | (n + \alpha) \mathfrak{a}
\text{ but }  \mathfrak{p}_n \nmid \prod_{m \le N_{j+1} 
\atop m \ne n} (m + \alpha) \mathfrak{a} \right\}  
$$
and
$$
\mathfrak{B}:= \left\{ n ~|~ N_j < n \le N_{j+1}, ~  n \not\in \mathfrak{A} \right\}. 
$$
By Cassels's lemma and the choice of $N_j$, we have $|\mathfrak{A}| \ge 27 M_j/ 50$. 

Note that if $\mathfrak{p} | \prod_{m \le N_{j+1}} (m + \alpha) \mathfrak{a}$,
then there are three possibilities, namely  $\mathfrak{p} |  \prod_{m \le N_j} (m + \alpha) \mathfrak{a}$
or $\mathfrak{p} = \mathfrak{p}_n$ for some $n \in \mathfrak{A}$
or $\mathfrak{p}$ is different from both of these types.
By induction hypothesis, we already know the values of $\chi(\mathfrak{p})$ 
when $\mathfrak{p}$  is of the first type and we define $\chi(\mathfrak{p}):= 1$ 
when $\mathfrak{p}$ is of the third type. Now we will define  
$\chi(\mathfrak{p})$ when $\mathfrak{p}$ is of the second type in such a way
that \eqref{eqalg} is true for $j+1$.

By the hypothesis of the theorem, we have
$$
\frac{f(n)}{f(m)} \le 1.15 \phantom{m} \text{ for all } n,m \in \N,
$$
and hence
$$
\frac{f(n)}{(n+\alpha)^{\sigma}} <  \frac{3 f(m)}{(m+\alpha)^{\sigma}}
$$
for any $n,m \in \mathfrak{A}$ as we have
\begin{equation*}
{(n+a)^{\sigma}\over (m+a)^{\sigma}} < 2 
\end{equation*}
for any $n,m \in \mathfrak{A}$. 
Since $|\mathfrak{A}| > 5$, it follows from \propref{prop3} 
that 
$$
\sum_{n \in \mathfrak{A}}\frac{f(n) \chi(n+\alpha)}{(n+\alpha)^{\sigma}} 
$$
can take any value $z$ with 
$$
|z| \le \sum_{n \in \mathfrak{A}}\frac{f(n)}{(n+\alpha)^{\sigma}}
=S_3, \text{ say}. 
$$
Write
\begin{equation*}
\Lambda := \sum_{n \le N_j} \frac{f(n) \chi(n+\alpha)}{(n+\alpha)^{\sigma}}  
              ~+~ \sum_{n \in \mathfrak{B}} \frac{f(n) \chi(n+\alpha)}{(n+\alpha)^{\sigma}} 
\end{equation*}
and set
\begin{eqnarray*}
S_1 &:=& \left| \sum_{n \le N_j} \frac{f(n) \chi(n+\alpha)}{(n+\alpha)^{\sigma}}\right| \\
S_2 &:=& \sum_{n \in \mathfrak{B}} \frac{f(n)}{(n+\alpha)^{\sigma}}. 
\end{eqnarray*}
Also set
\begin{equation*}
z := 
\begin{cases} 
 - \Lambda & \text{ if  $0 < |\Lambda| \le S_3$},\\
 -  S_3 \Lambda / |\Lambda| & \text{ if $|\Lambda| > S_3$}, \\
  0   & \text{ if  $\Lambda = 0$.}
 \end{cases} 
\end{equation*}
Then by appropriate choice of $\chi(n+\alpha)$ for $n \in \mathfrak{A}$,
we have
\begin{eqnarray*}
\left| \sum_{n \le N_{j+1}} \frac{f(n) \chi(n+\alpha)}{(n+\alpha)^{\sigma}}   \right| 
\le \text{ max} \left\{ 0, ~S_1 + S_2 - S_3 \right\}.
\end{eqnarray*}
Since $|\mathfrak{B} | \le 23 M_j/50$ and $|\mathfrak{A}| \ge 27M_j /50$, we have
\begin{eqnarray*}
\frac{S_3}{S_2} &\ge&  \frac{27}{23c}
\frac{(N_j + \alpha)^{\sigma}}{(N_{j+1} + \alpha)^{\sigma}}\\
& > & \frac{27 \times 10^7 \times 10^7}{23(10^7 +11)^2\times 1.15} ~>~ \frac{101}{99}.
\end{eqnarray*}
This implies that
\begin{equation}\label{new}
100 (S_3 - S_2) > S_3 + S_2.
\end{equation}
Set 
$$
S_4:= \sum_{n > N_{j+1}} \frac{f(n)}{(n + \alpha)^{\sigma}}.
$$
By induction,
$$
S_1 ~<~ 10^{-2} (S_3 + S_2 + S_4).
$$
Thus by \eqref{new}, we have
$$
S_1 + S_2 - S_3 < 10^{-2} S_4.
$$
This proves \eqref{eqalg} by induction
and hence we have \eqref{algebraic}. 
 This proves the theorem.

\bigskip
\noindent
{\bf Acknowledgments.} The authors thank R. Balasubramanian,  M. Ram Murty, J. Oesterl\'e, P. Rath,   
B. Saha and E. Saha for helpful discussions.  
They would also like to thank the referee for valuable suggestions.


\smallskip

\begin{thebibliography}{100}
\bibitem{BG}
E. Bombieri and A. Ghosh,  {\em On the Davenport-Heilbronn function}, 
Uspekhi Mat. Nauk {\bf 66} (2011), no. 2, 15--66.
\bibitem{BT}
A. R. Booker and F. Thorne, {\em Zeros of $L$-functions outside the critical strip}, \linebreak
arXiv: 1306.6362. 
\bibitem{JC}
J. W. S. Cassels, {\em Footnote to a note of Davenport and Heilbronn}, 
J. London Math. Soc. {\bf 36} (1961), 177--184.
\bibitem{KC}
K. Chandrasekharan, {\em Introduction to analytic number theory}, 
Grundlehren Math. Wiss. 148, Springer, New York, 1968. 
\bibitem{DH}
H. Davenport and H. Heilbronn, {\em On the zeros of certain Dirichlet series}, 
Proc. London Math. Soc. {\bf 11} (1936), 181--185.
\bibitem{SW}
E. Saias and A. Weingartner, {\em Zeros of Dirichlet series with 
periodic coefficients}, Acta Arith {\bf 140} (2009), no. 4, 335--344. 
\end{thebibliography}
\end{document}